\documentclass{article}
\usepackage[onehalfspacing]{setspace}
\usepackage[margin=.6in]{geometry}
\usepackage{amsfonts}
\usepackage{amsmath}
\usepackage{graphicx}
\usepackage{mathrsfs}
\usepackage{amssymb}
\usepackage{amsthm}
\usepackage{tikz}
\usepackage{tikz-cd}
\usepackage{listings}
\usepackage{color}
\usepackage{enumitem}
\usepackage{imakeidx}
\usepackage[colorinlistoftodos]{todonotes}
\usepackage{upgreek}
\usepackage{hyperref}
\usepackage{mdframed}
\usepackage{tocbibind}
\usepackage{pgfplots}
\usepackage{marvosym}

\allowdisplaybreaks
\newcommand{\inv}[1]{#1^{-1}}

\renewcommand{\d}{\,\mathrm d} 

\newcommand{\Ith}[2]{{#1}^{\left(#2\right)}}
\newcommand{\ith}[1]{\Ith{#1}i}

\newcommand{\mc}{\mathcal}

\DeclareMathOperator{\Z}{\mathbb Z}
\DeclareMathOperator{\R}{\mathbb R}
\DeclareMathOperator{\C}{\mathbb C}

\DeclareMathOperator{\N}{\mathbb N}
\DeclareMathOperator{\eps}{\varepsilon}

\renewcommand{\tau}{\uptau}

\newcommand{\parens}[1]{\left(#1\right)}
\newcommand{\brackets}[1]{\left\{#1\right\}}
\newcommand{\sqbracks}[1]{\left[#1\right]}

\newcommand*{\rom}[1]{\textup{\uppercase\expandafter{\romannumeral#1}}}
\newcommand{\tbf}{\textbf}

\newcommand{\sm}{\setminus}

\renewcommand{\ast}[1]{#1^*}

\renewcommand{\ast}[1]{#1^*}

\renewcommand{\Re}{\mathrm{Re}\,}

\makeatletter
\def\renewtheorem#1{%
	\expandafter\let\csname#1\endcsname\relax
	\expandafter\let\csname c@#1\endcsname\relax
	\gdef\renewtheorem@envname{#1}
	\renewtheorem@secpar
}
\def\renewtheorem@secpar{\@ifnextchar[{\renewtheorem@numberedlike}{\renewtheorem@nonumberedlike}}
\def\renewtheorem@numberedlike[#1]#2{\newtheorem{\renewtheorem@envname}[#1]{#2}}
\def\renewtheorem@nonumberedlike#1{  
	\def\renewtheorem@caption{#1}
	\edef\renewtheorem@nowithin{\noexpand\newtheorem{\renewtheorem@envname}{\renewtheorem@caption}}
	\renewtheorem@thirdpar
}
\def\renewtheorem@thirdpar{\@ifnextchar[{\renewtheorem@within}{\renewtheorem@nowithin}}
\def\renewtheorem@within[#1]{\renewtheorem@nowithin[#1]}
\makeatother

\DeclareDocumentEnvironment{MyFrame}{O{1cm}O{0.4pt}O{0.8cm}O{black}O{3}O{2ex}}
{\par\hfill\rlap{%
		\bgroup\color{#4}%
		\hskip-\dimexpr#1-#3\relax\rule{#1}{#2}%
		\hskip-\dimexpr#1/#5\relax\rule[-\dimexpr#1-\dimexpr#1/#5\relax]{#2}{#1}%
		\egroup
	}%
	\vskip-\dimexpr#1/#5+\dimexpr#1/#5-#6\relax%
}
{\par\nobreak\offinterlineskip\vskip-\dimexpr#1/#5+\dimexpr#1/#5-#6\relax\noindent%
	\hskip-#3\bgroup\color{#4}%
	\rule{#1}{#2}\hskip-\dimexpr#1-\dimexpr#1/#5-#2\relax%
	\rule[-\dimexpr#1/#5-#2\relax]{#2}{#1}\egroup\par
}

\newtheorem{innercustomgeneric}{\customgenericname}
\providecommand{\customgenericname}{}
\newcommand{\newcustomtheorem}[2]{%
	\newenvironment{#1}[1]
	{%
		\renewcommand\customgenericname{#2}%
		\renewcommand\theinnercustomgeneric{##1}%
		\innercustomgeneric
	}
	{\endinnercustomgeneric}
}

\theoremstyle{plain}
\newtheorem{thm}{Theorem}
\newtheorem{lemma}[thm]{Lemma}
\newtheorem{cor}[thm]{Corollary}
\newtheorem{conject}[thm]{Conjecture}
\newtheorem{prop}[thm]{Proposition}

\newcustomtheorem{Prob}{Problem}
\newcustomtheorem{Exc}{Theoretical Exercise}

\theoremstyle{definition}
\newtheorem{defn}[thm]{Definition}

\newtheorem{notn}[thm]{Notation}

\theoremstyle{remark}
\newtheorem{rem}[thm]{Remark}

\newcommand{\xsspacing}[0]{\hspace*{10pt}}
\newcommand{\sspacing}[0]{\hspace*{25pt}}

\newcommand{\Text}[2]{\text{#2 #1 #2}}

\newcommand{\thought}[1]{\todo[color=green!50]{#1}}

\usepackage{soul}

\renewcommand{\S}{\mathbb S}

\theoremstyle{plain}
\renewtheorem{thm}{Theorem}[section]
\renewtheorem{cor}[thm]{Corollary}
\renewtheorem{prop}[thm]{Proposition}

\renewtheorem{conj}[thm]{Conjecture}

\theoremstyle{definition}
\renewtheorem{defn}[thm]{Definition}

\renewtheorem{notn}[thm]{Notation}

\theoremstyle{remark}
\renewtheorem{rem}[thm]{Remark}

\title{On Gaps in the Closures of Images of Divisor Functions}

\author{Niven Achenjang\footnote{Stanford University, \textit{niven@stanford.edu}}\, and Aaron Berger\footnote{MIT, \textit{bergera@mit.edu}}}

\date{}

\begin{document}

\maketitle
\begin{abstract}
	Given a complex number $c$, define the divisor function $\sigma_c:\N\to\C$ by $\sigma_c(n)=\sum_{d\mid n}d^c$. In this paper, we look at $\overline{\sigma_{-r}(\N)}$, the topological closures of the image of $\sigma_{-r}$, when $r>1$. We exhibit new lower bounds on the number of connected components of $\overline{\sigma_{-r}(\N)}$, bringing this bound from linear in $r$ to exponential. Finally, we discuss the general structure of gaps of $\overline{\sigma_{-r}(\N)}$ in order to work towards a possible monotonicity result.
\end{abstract}

\section{Introduction}
\indent
Our main objects of study in this paper will be divisor functions. Given a complex number $c\in\C$, the divisor function $\sigma_c:\N\to\C$ is given by
$$\sigma_c(n)=\sum_{d\mid n}d^c,$$
where $\N=\{1,2,\dots\}$ is the set of positive integers. Laatsch studied the set $\sigma_{-1}(\N)$ in 1986 \cite{laat}, showing that it is dense in $[1,\infty)$. Motivated by this, Defant \cite{coin1} began the study of topological properties of $\overline{\sigma_{-r}(\N)}$ for a real parameter $r>1$. In particular, he showed that $\overline{\sigma_{-r}(\N)}$ is connected for $r$ in the range $(0,\eta]$ where $\eta\approx1.88779$ is a constant, now called the Defantstant \cite{zubr}, satisfying
$$\frac{2^\eta}{2^\eta-1}\frac{3^\eta+1}{3^\eta-1}=\zeta(\eta).$$
\indent
Sanna \cite{sann} gave an algorithm for computing $\overline{\sigma_{-r}(\N)}$ for a given $r$, and used this algorithm to show that $\overline{\sigma_{-r}(\N)}$ always has finitely many connected components. Zubrilina \cite{zubr} studied the number of connected components of $\overline{\sigma_{-r}(\N)}$ in more detail. In particular, letting $C_r$ denote the number of connected components of $\overline{\sigma_{-r}(\N)}$, she showed that 
$$\pi(r)+1\le C_r\le\frac12\exp\sqbracks{\frac12\frac{r^{20/9}}{(\log r)^{29/9}}\parens{1+\frac{\log\log r}{\log r-\log\log r}+\frac{\mc O(1)}{\log r}}}$$
where $\pi(r)$ is the number of primes at most $r$. In addition, she showed that $C_r$ does not take on all finite values; in particular, she showed that $C_r\neq4$ for all real $r$. Such numbers are now called \tbf{Zubrilina numbers} \cite{coin3}.

\indent
Let $p_m$ denote the $m$th prime number. Work in this field is reliant upon the notion of \tbf{$r$-mighty primes}, which are primes $p_m$ such that
$$1+\frac1{p_m^r}>\prod_{t=m+1}^\infty\frac1{1-p_t^{-r}}.$$
Furthermore, for understanding the behavior of $C_r$, it has proven useful to study the gaps of $\overline{\sigma_{-r}(\N)}$, where by a gap we mean a bounded connected component of $\R\sm\overline{\sigma_{-r}(\N)}$. In this paper, we extend Zubrilina's work by showing that 6 is also a Zubrilina number and by improving her lower bound for $C_r$ in both the asymptotic and small-$r$ cases. At the end, we will look at the general structure of gaps of $\overline{\sigma_{-r}(\N)}$ before finishing with some open problems. 

\section{The Effect of Taking the Closure}
In this section, we will provide an alternate description of $\overline{\sigma_{-r}(\N)}$ that is simpler to work with later on because it avoids the need to take a closure. This new description results from replacing the domain of $\sigma_{-r}$ with the larger set of ``supernatural'' numbers, which allow for infinitely many prime factors.
\begin{defn}[\cite{Steinitz1910}]
	A \tbf{Steinitz} (or \tbf{supernatural}) \tbf{number} is a formal product
	$$n=\prod_{p\text{ prime}}p^{\alpha_p}$$
	where $\alpha_p\in\Z_{\ge0}\cup\,\{\infty\}$ for all primes $p$. We extend the usual $p$-adic valuation to such numbers by setting $v_p(n)=\alpha_p$. Finally, let $\S$ denote the set of all Steinitz numbers.\footnote{It seems there is no standard notation in the literature for denoting this set.}
\end{defn}
\begin{rem}
	For $c\in\C$ with $\Re(c)<-1$, we can naturally extend $\sigma_c$ to a function $\S\to\C$ by setting
	$$\sigma_c(p^\infty)=\lim_{n\to\infty}\sigma_c(p^n)=\frac1{1-p^c}.$$
	We still require $\sigma_c$ to be multiplicative on $\S$, so, for example, we have $\sigma_c(2^\infty\cdot3^2)=\sigma_c(2^\infty)\sigma_c(3^2)=\parens{\frac1{1-2^c}}\parens{\frac{1-3^{3c}}{1-3^c}}$.
\end{rem}
The utility of introducing these Steinitz numbers is demonstrated in the following proposition.
\begin{prop}\label{stein}
	Fix a complex number $c$ with $\Re(c)<-1$. Then $$\overline{\sigma_c(\N)} = \sigma_c(\S).$$
\end{prop}
\begin{proof}
	$(\supseteq)$ Given $n\in\S$, let $A=\{m\in\N:v_{p_m}(n)<\infty\}$ and let $B=\N\sm A$. For convenience, write $\alpha_m=v_{p_m}(n)$ for all $m\in A$. Then,
	$$\sigma_{c}(n)=\lim_{k\to\infty}\sigma_c\parens{\parens{\prod_{\substack{m\in A\\m\le k}}p_m^{\alpha_m}}\parens{\prod_{\substack{m\in B\\m\le k}}^kp_m^k}}\in\overline{\sigma_{c}(\N)}.$$
	($\subseteq$) Let
	$$a_i=\prod_{m=1}^\infty p_m^{\ith\alpha_m}\in\N$$
	be a sequence of natural numbers such that $\sigma_c(a_i)\to L\in\overline{\sigma_{-r}(\N)}$ as $i\to\infty$. We will inductively construct
	$$a=\prod_{m=1}^\infty p_m^{\alpha_m}\in\S$$
	such that $\sigma_c(a)=L$. First, to find $\alpha_1$, note that there's a monotone increasing subsequence of $\ith\alpha_1$. By replacing $a_i$ with the subsequence of elements corresponding to this subsequence of $\ith\alpha_1$, we can assume that $\ith\alpha_1$ is monotone increasing. Furthermore, after doing this, we still have that $L=\lim\sigma_c(a_i)$ and can set $\alpha_1=\lim\ith\alpha_1\in\Z_{\ge0}\cup\,\{\infty\}$. Now, inductively repeat this process, replacing $\ith\alpha_m$ with a monotone increasing subsequence and setting $\alpha_m=\lim{\ith\alpha_m}$. By construction, we get
	\begin{align*}
	\sigma_{-r}\parens{\prod_{m=1}^\infty p_m^{\alpha_m}}=L,
	\end{align*}
	as desired.
\end{proof}
Because of the above theorem, for the remainder of the paper, we can work directly with $\sigma_{-r}(\S)$ instead of $\overline{\sigma_{-r}(\N)}$.
\begin{rem}
	We can restate the definition of $r$-mighty primes by saying that $p_m$ is $r$-mighty if
	$$\sigma_{-r}(p_m)>\sigma_{-r}\parens{\prod_{t=m+1}^\infty p_t^\infty}.$$
\end{rem}

\section{A New Class of Gaps of $\sigma_{-r}(\S)$}\label{gap sec}
In this section, we will exhibit a new class of gaps of $\sigma_{-r}(\S)$. We first need to introduce some notation.
\begin{notn} 
    Let
    \begin{itemize}
        \item $N_r$ denote the number of $r$-mighty primes.
        \item $r_p=\inf\brackets{s>1:p\text{ is $s$-mighty}}$.
        \item
        $$u_m(r)=\sigma_{-r}\parens{\prod_{t=m+1}^\infty p_t^\infty}=\prod_{t=m+1}^\infty\frac1{1-p_t^{-r}}.$$
    \end{itemize}
	We will often not make the $r$ explicit, writing $u_m$ instead of $u_m(r)$.
\end{notn}
Zubrilina \cite{zubr} bounded $C_r$ below by showing that $\parens{u_m,\sigma_{-r}(p_m)}$ is a gap of $\sigma_{-r}(\S)$ when $p_m$ is $r$-mighty, allowing her to conclude that $C_r\ge1+N_r$. We will similarly show that for $r$-mighty primes $p_m,q$ with $p_m>q^2$,
$$\parens{\sigma_{-r}(q)u_m,\sigma_{-r}(qp_m))}$$
is a gap of $\sigma_{-r}(\S)$. Zubrilina \cite{zubr} showed that $r_p$ is always finite and that $p$ is $r$-mighty if and only if $r>r_p$. Furthermore, she showed that
\begin{equation}\label{mightyorder}
	r_3<r_2<r_5<r_7<r_p\text{ for primes }p>7.
\end{equation}
Using a computer, one can verify that
$$\begin{matrix}
1.8 &<& r_3 &<& 1.9\\
1.9 &<& r_2 &<& 2\\
2.2 &<& r_5 &<& 2.3\\
2.4 &<& r_7 &<& 2.5
\end{matrix}.$$
Now, we establish two lemmas used in the proof of this section's main theorem.
\begin{lemma}\label{divqfull}
	Let $q$ be an $r$-mighty prime, let $p_m>q^2$ also be prime. Fix some $r\ge2$ and some $n\in\S$ such that
	$$\sigma_{-r}(q) < \sigma_{-r}(n) < \sigma_{-r}(qp_m).$$
	Then, $q\mid n$, and $q$ is the smallest prime dividing $n$.
\end{lemma}
\begin{proof}
	It's easy to see that some prime $p_k\le q$ must divide $n$ since otherwise we would have (where $q=p_{m(q)}$)
	$$\sigma_{-r}(n)\le\sigma_{-r}\parens{\prod_{k=m(q)+1}^\infty p_k^\infty}=u_{m(q)}<\sigma_{-r}(q).$$
	Suppose that $p_k<q$. We will show that
	$$\sigma_{-r}(p_k)=1+p_k^{-r}>1+q^{-r}+p_m^{-r}+(qp_m)^{-r}=\sigma_{-r}(qp_m).$$
	Since $p_m>q^2$ and $p_k\le(q-1)$, it suffices to show that $(q-1)^{-r}>q^{-r}+q^{-2r}+q^{-3r}$. By Bernoulli's inequality,
	$$(q-1)^{-r}=q^{-r}(1-\inv q)^{-r}\ge q^{-r}(1+r\inv q)=q^{-r}+rq^{-r-1},$$
	so it suffices to show that $q^{-r}+rq^{-r-1}>q^{-r}+q^{-2r}+q^{-3r}$. Simplifying this inequality yields $rq^{-r-1}>q^{-2r}+q^{-3r}$, which holds since $r\ge2$. Thus, we have shown that
	\begin{align*}
	    \sigma_{-r}(n)\ge\sigma_{-r}(p_k)>\sigma_{-r}(qp_m).&\qedhere
	\end{align*}
\end{proof}
\begin{cor}\label{divq}
	Let $p_m,q$ be $r$-mighty primes such that $p_m>q^2$. Pick some $n\in\N$ such that $\sigma_{-r}(n)\in \parens{\sigma_{-r}(q)u_m,\sigma_{-r}(qp_m)}$. Then, $q\mid n$.
\end{cor}
\begin{proof}
	Because $p_m$ is $r$-mighty, it follows from Zubrilina's result that $r>r_{p_m}$. We know that $p_m\ge q^2\ge2^2$. By (\ref{mightyorder}), this means that $r_{p_m}\ge r_5$. Using Zubrilina's result once more, we see that $5$ is $r$-mighty, so $r>r_5>2$. Thus, we can apply Lemma \ref{divqfull}
\end{proof}
\begin{lemma}
	Fix $r>2.2$ and let $p$ and $q$ be primes such that $p>q^2$. Then $\sigma_{-r}(qp)<\sigma_{-r}(q^2)$.
\end{lemma}
\begin{proof}
    We seek to show that $1+q^{-r}+p^{-r}+(qp)^{-r} < 1+q^{-r}+q^{-2r}$ which simplifies to $p^{-r}+(qp)^{-r}<q^{-2r}$. Since $p\ge q^2+1$, it suffices to show that 
    $$(q^2+1)^{-r}+(q^3+q)^{-r}<q^{2r}.$$
    It is straightforward to verify that $(1+x)^r\le1+(2^r-1)x$ when $x\in[0,1]$ and $r\not\in(0,1)$. Using this, we see that
    $$(q^2+1)^{-r}+(q^3+q)^{-r}\le q^{-2r}+q^{-2r-2}(2^{-r}-1)+q^{-3r}+q^{-3r-2}(2^{-r}-1)$$
    so it suffices to show that $q^{-2r}+q^{-2r-2}(2^{-r}-1)+q^{-3r}+q^{-3r-2}(2^{-r}-1)<q^{-2r}$. Simplifying this inequality yields $(2^{-r}-1)(q^{-2r-2}+q^{-3r-2})<-q^{-3r}$. Since $r>2$, we have $2^{-r}<1/4$, so
    we only need to show that
    $$q^{-2r-2}+q^{-3r-2}>\frac43q^{-3r}.$$
    Dividing by $q^{-3r}$ gives $q^{r-2}+q^{-2}>\frac43$ which holds since $q\ge2$, $r\ge2.2$, and $2^{0.2}+2^{-2}>1.39>4/3$.
\end{proof}
\begin{cor}\label{notq2}
    Let $p$ and $q$ be $r$-mighty primes such that $p>q^2$. Then $\sigma_{-r}(qp)<\sigma_{-r}(q^2)$.
\end{cor}
We can now prove this section's main result.
\begin{thm}\label{newgap}
	Let $p_m,q$ be $r$-mighty primes with $q^2<p_m$. Then,
	$$G := \parens{\sigma_{-r}(q)u_m,\sigma_{-r}(qp_m)}$$
	is a gap of $\sigma_{-r}(\S)$.
\end{thm}
\begin{proof}
	Note that the endpoints of $G$ are in $\sigma_{-r}(\S)$ and that
	$$u_m\sigma_{-r}(q)<\sigma_{-r}(qp_m)$$
	because $u_m<\sigma_{-r}(p_m)$. Hence, we only need to show that $\sigma_{-r}(\N)\cap G=\emptyset$. Fix some $n\in\N$ such that $\sigma_{-r}(n)\in G$. By Corollary \ref{divq}, we must have $q\mid n$, so we can write $n=qs$. Because $p_m>q^2$, Corollary \ref{notq2} tells us that $\sigma_{-r}(qp_m)<\sigma_{-r}(q^2)$, which means that $q\nmid s$. Thus, $\sigma_{-r}(n)=\sigma_{-r}(q)\sigma_{-r}(s)$ and
	$$u_m\sigma_{-r}(q) <\sigma_{-r}(q)\sigma_{-r}(s)<\sigma_{-r}(qp_m)\iff u_m<\sigma_{-r}(s)<\sigma_{-r}(p_m).$$
    We already know that $(u_m,\sigma_{-r}(p_m))$ is a gap, so we are done.
\end{proof}
\begin{cor}\label{uglybound}
    Let $S_r$ denote the set of $r$-mighty primes. We have
	$$C_r\ge1+\sum_{q\in S_r\cup\{1\}}\#\brackets{p\in S_r:p>q^2}.$$
\end{cor}
\begin{rem}
	The inequality in Corollary \ref{uglybound} is useful when $r$ is small. An asymptotically better lower bound is derived in Section \ref{asym low bound}.
\end{rem}

\section{6 is a Zubrilina Number}
In this section we show that $6$ is a Zubrilina number. In other words, $C_r\neq6$ for all $r$. Showing this will involve an application of Sanna's algorithm, so we restate his main theorems below.
\begin{notn}
	Let
	$$\S_j=\brackets{n\in\S\mid v_p(n)>0\implies p>p_j}.$$
\end{notn}
\begin{notn}
	Let $L_r$ denote the index of the largest $r$-mighty prime, so $p_{L_r}$ is $r$-mighty, but $m>L_r$ implies that $p_m$ is not $r$-mighty. If there are no $r$-mighty primes, then we set $L_r=0$.
\end{notn}
\begin{thm}[\cite{sann}, Lemma 2.2]
	$L_r$ is finite for any $r>1$.
\end{thm}
\begin{thm}[\cite{sann}, Lemma 2.3]\label{interval}\footnote{Sanna's original formulation is that $\overline{\sigma_{-r}(\N_{L_r})}=[1,u_{L_r}]$, where $\N_j=\S_j\cap\N$.}
	$$\sigma_{-r}(\S_{L_r})=\sqbracks{1,u_{L_r}}.$$
\end{thm}
\begin{thm}[\cite{sann}, Lemma 2.4]
	$$\sigma_{-r}(\S_K)=\bigcup_{i\in\Z_{\ge0}\cup\,\{\infty\}}\sigma_{-r}(p_{K+1}^i)\cdot\sigma_{-r}(\S_{K+1})$$
	where we write $a\cdot X=\{ax\mid x\in X\}$ for a number $a$ and set $X$.
\end{thm}
\begin{thm}[\cite{sann}, Lemma 2.5]
	Let $I=[a,b]$, and fix a prime $p$. Let $t$ be the least non-negative integer such that
	$$\frac{\sigma_{-r}(p^{t+1})}{\sigma_{-r}(p^t)}\le\frac ba.$$
	Then, the following is a decomposition into disjoint intervals
	$$\bigcup_{i\in\Z_{\ge0}\cup\,\{\infty\}}\sigma_{-r}(p^i)\cdot I=\parens{\bigsqcup_{0\le i<t}\sigma_{-r}(p^i)\cdot I}\sqcup\sqbracks{a\sigma_{-r}(p^t),b\sigma_{-r}(p^\infty)}.$$
\end{thm}
Sanna used these theorems to construct the following backwards induction algorithm for calculating $\sigma_{-r}(\S_0)=\sigma_{-r}(\S)$ from $\sigma_{-r}(\S_{L_r})$:
\begin{enumerate}
	\item We know at the start that $\sigma_{-r}(\S_{L_r})=\sqbracks{1,u_{L_r}}$.
	\item Suppose we have $\sigma_{-r}(\S_K)=\bigcup_{j\in J}I_j$ for some $K\in\N$ and index set $J$. Write $I_j=[a_j,b_j]$ and let $t_j$ be the least non-negative integer such that
	$$\frac{\sigma_{-r}(p_K^{t_j+1})}{\sigma_{-r}(p_K^{t_j})}\le\frac{b_j}{a_j}.$$
	\item We have
	$$\sigma_{-r}(\S_{K-1})=\bigcup_{j\in J}\parens{\bigsqcup_{0\le i<t_j}\sigma_{-r}(p_K^i)\cdot I_j\sqcup\sqbracks{a_j\sigma_{-r}(p_K^{t_j}),b_j\sigma_{-r}(p_K^\infty)}}.$$
\end{enumerate}
We will apply this algorithm to bound $C_r$ when $L_r \le 3$:
\begin{lemma}\label{Lr3}
	If $L_r\le3$, then $C_r\le5$.
\end{lemma}
\begin{proof}
	Defant \cite{coin1} showed that $C_r=1$ when $L_r=0$ while Zubrilina \cite{zubr} showed that $C_r=3$ when  $L_r=2$. $L_r$ cannot equal 1 by (\ref{mightyorder}), so we may assume that $L_r=3$. Note that the $r$-mighty primes must be $2,3$ and $5$. We apply Sanna's algorithm:
	\begin{enumerate}
		\item We start with $\sigma_{-r}(\S_3)=[1,u_3]$.
		\item We now seek the least $t\in\Z_{\ge0}$ such that
		$$\frac{\sigma_{-r}(5^{t+1})}{\sigma_{-r}(5^t)}<\sigma_{-r}(u_3).$$
		Using that $2.2<r_5<r<r_7<2.5$, a computer calculation shows that $t=1$. Hence,
		$$\sigma_{-r}(\S_2)=\sqbracks{1,u_3}\cup\sqbracks{1+\frac1{5^r},u_2}.$$
		\item Now, we find that the smallest $t_1,t_2\in\Z_{\ge0}$ such that
		$$\frac{\sigma_{-r}(3^{t_1+1})}{\sigma_{-r}(3^{t_1})}<u_3\Text{and}{\sspacing}\frac{\sigma_{-r}(3^{t_2+1})}{\sigma_{-r}(3^{t_2})}<\frac{u_2}{1+\frac1{5^r}}$$
		are $t_1=t_2=1$. This gives
		$$\sigma_{-r}(\S_1)=\sqbracks{1,u_3}\cup\sqbracks{1+\frac1{5^r},u_2}\cup\sqbracks{1+\frac1{3^r},\frac{u_3}{1-3^{-r}}}\cup\sqbracks{\parens{1+\frac1{5^r}}\parens{1+\frac1{3^r}},u_1}.$$
		However, using Mathematica \cite{math}, we discover that for $r<2.5$, we have
		$$\parens{1+\frac1{5^r}}\parens{1+\frac1{3^r}}<\frac{u_3}{1-3^{-r}}.$$
		This allows us to coalesce the last two intervals, writing
		$$\sigma_{-r}(\S_1)=\sqbracks{1,u_3}\cup\sqbracks{1+\frac1{5^r},u_2}\cup\sqbracks{1+\frac1{3^r},u_1}.$$
		\item Last step. We need the smallest $t_1,t_2,t_3\in\Z_{\ge0}$ such that
		$$\frac{\sigma_{-r}(2^{t_1+1})}{\sigma_{-r}(2^{t_1})}< u_3\Text{and}{\xsspacing}\frac{\sigma_{-r}(2^{t_2+1})}{\sigma_{-r}(2^{t_2})}<\frac{u_2}{1+\frac1{5^r}}\Text{and}{\xsspacing}\frac{\sigma_{-r}(2^{t_3+1})}{\sigma_{-r}(2^{t_3})}<\frac{u_1}{1+\frac1{3^r}}.$$
		These are $t_1=2$, $t_2=2$, and $t_3=1$. This gives the 
		decomposition
		\begin{align*}
		\overline{\sigma_{-r}(\N)} = &\phantom{\cup}\sqbracks{1,u_3}\cup\sqbracks{1+\frac1{5^r},u_2}\cup\sqbracks{1+\frac1{3^r},u_1}\cup\sqbracks{\parens{1+\frac1{2^r}}\parens{1+\frac1{3^r}},\zeta(r)}\\
		&\cup\sqbracks{\parens{1+\frac1{2^r}}\parens{1+\frac1{5^r}},\parens{1+\frac1{2^r}}u_2}\cup\sqbracks{1+\frac1{2^r},\parens{1+\frac1{2^r}}u_3}\\
		&\cup\sqbracks{\parens{1+\frac1{2^r}+\frac1{4^r}}\parens{1+\frac1{5^r}},\frac{u_2}{1-2^{-r}}}\cup\sqbracks{1+\frac1{2^r}+\frac1{4^r},\frac{u_3}{1-2^{-r}}}.
		\end{align*}
		Using the fact that $r<2.5$ along with Mathematica \cite{math} allows us to combine the the last two intervals with the 4th and 5th intervals, resulting in
		$$\sigma_{-r}(\S)=\sqbracks{1,u_3}\cup\sqbracks{1+\frac1{5^r},u_2}\cup\sqbracks{1+\frac1{3^r},u_1}\cup\sqbracks{1+\frac1{2^r},\parens{1+\frac1{2^r}}u_3}\cup\sqbracks{\parens{1+\frac1{2^r}}\parens{1+\frac1{5^r}},\zeta(r)},$$
		which is a union of 5 intervals.
		\qedhere
	\end{enumerate}
\end{proof}
\begin{lemma}\label{Lr4}
	If $L_r\ge4$, then $C_r\ge7$.
\end{lemma}
\begin{proof}
	By (\ref{mightyorder}), the four $r$-mighty primes must be $2,3,5$, and $7$. Since $5$ and $7$ are greater than $2^2$, Corollary \ref{uglybound} shows that $C_r\ge 7$.
\end{proof}
Combining Lemmas \ref{Lr3} and \ref{Lr4}, we conclude:
\begin{thm}
	6 is a Zubrilina number.
\end{thm}

\section{An Exponential Lower Bound}\label{asym low bound}
In this section we show that $C_r \ge 2^{\pi(r - 1)}$. This makes the upper and lower bounds on $C_r$ both exponential in $r$.

\begin{lemma}\label{Less than r}
	If $1 \leq d \leq r - 1$, then $\sum_{n = d+1}^\infty n^{-r} < d^{-r}$.
\end{lemma}
\begin{proof}
	Since $n^{-r}$ is decreasing in $n$, we have $$\sum_{n = d+1}^\infty n^{-r} < \int_d^\infty x^{-r} dx = \frac{1}{r-1} d^{-r+1} \leq d^{-r}.$$
\end{proof}
\begin{thm}
	Let $n$ be a positive integer such that $p^k\le r-1$ whenever $p^k$ is a prime power dividing $n$. Then, the point $\sigma_{-r}(n)$ is the right endpoint of a gap in $\sigma_{-r}(\S)$.
\end{thm}
\begin{proof}
	Let $m \neq n$ be a positive integer. If $p^k \mid m$ whenever $p^k \mid n$, then we have $n \mid m$ and so $\sigma_{-r}(m) > \sigma_{-r}(n)$. Otherwise, there is some prime $p$ and integer $k$ such that $p^k \mid n$ but $p^k \nmid m$. By assumption, this means $p^k \le r - 1$. Then the set of integers $d \le r - 1$ that divide exactly one of $m$ and $n$ is nonempty. Let $d'$ be the smallest such integer. If $d' \mid m$, then 
	$$\sigma_{-r}(m) \ge (d')^{-r} + \sum_{x < d', x \mid m} x^{-r}  = 
	(d')^{-r} + \sum_{x < d', x \mid n} x^{-r} >
	\sum_{x > d'} x^{-r}+ \sum_{x < d', x \mid n}x^{-r} \ge 
	\sigma_{-r}(n),$$
	where the second inequality follows from Lemma \ref{Less than r}. If $d' \mid n$, we have
	$$\sigma_{-r}(n) \ge (d')^{-r} + \sum_{x < d', x \mid n} x^{-r}  =
	(d')^{-r} + \sum_{x < d', x \mid m} x^{-r} >
	\sum_{x > d'} x^{-r}+ \sum_{x < d', x \mid m}x^{-r} \ge 
	\sigma_{-r}(m).$$
    We can strengthen the second inequality by saying that $\sigma_{-r}(n)\ge\sigma_{-r}(m)+\eps$ where 
    $$\eps=\min_{\substack{1\le d\le r-1\\d\in\N}}\parens{d^{-r}-\sum_{k=d+1}^\infty k^{-r}}$$ 
    is a positive constant independent of $n$ and $m$. Concluding, we have that for all $m$, either $\sigma_{-r}(m) \ge\sigma_{-r}(n)$ or $\sigma_{-r}(m) \le \sigma_{-r}(n) - \eps$. We conclude that $n$ is the right endpoint of a gap that is of length at least $\eps > 0$.
\end{proof}
\begin{cor}
	If $n$ is a product of primes less than $r - 1$, then $\sigma_{-r}(n)$ is the right endpoint of a gap of $\sigma_{-r}(\S)$.
\end{cor}
\begin{cor}
	For a given $r>1$, the number of gaps of $\sigma_{-r}(\S)$ is at least $2^{\pi(r-1)} \approx 2^{r/\log(r)}$.
\end{cor}
\begin{rem}
    The gaps found in Section \ref{gap sec} along with the ones found here do not, in general, account for all gaps of $\sigma_{-r}(\S)$. For $r=3$, they predict a lower bound of $C_r\ge7$, but Sanna \cite{sann} produced a plot of $\sigma_{-r}(\S)$ for various values of $r$ which shows that, in fact, $C_r\ge14$ for $r\approx3$.
\end{rem}

\section{The Structure of General Gaps}
In this section, we study the structure of general gaps of $\sigma_{-r}(\S)$. We will show that the preimages in $\S$ of the endpoints of a gap of $\sigma_{-r}(\S)$ are fairly constrained. In particular, we will prove a series of lemmas whose statements are collected together in the following theorem.
\begin{notn}
	Let $P_r=p_{L_r+1}$ denote the smallest prime larger than all $r$-mighty primes.
\end{notn}
\begin{thm}\label{struct}
	Let $\left(\sigma_{-r}(a),\sigma_{-r}(b)\right)$ with $a,b\in\S$ be a gap of $\sigma_{-r}(\S)$.
	\begin{itemize}
		\item If $p\ge P_r$, then $$v_p(a)=\infty\text{ and }v_p(b)=0.$$
		\item If $p<P_r$ and $v_p(a)<\infty$, then
		$$v_p(a)<\frac{\log P_r}{\log p}-1.$$
		\item If $p<P_r$, then
		$$v_p(b)<\frac{\log P_r}{\log p}.$$
	\end{itemize}
\end{thm}
\begin{cor}
	Under the same assumptions as in the above theorem, as long as $r$ is sufficiently large, if $v_p(a)>1$ for some prime $p$ then $p$ is $r$-mighty.
\end{cor}
\begin{proof}
	Zubrilina \cite{zubr} shows that for $r$ sufficiently large, $P_r<r^{20/9}$. She also shows that $p<r\implies p$ is $r$-mighty. Fix some prime $p$ such that $v_p(a)>1$. The above theorem shows that $P_r>p^3$. Thus, $p<\parens{r^{20/9}}^{1/3}<r$, so $p$ is $r$-mighty.
\end{proof}
We now prove the various parts of Theorem \ref{struct}. Throughout the lemmas, assume that $\parens{\sigma_{-r}(a),\sigma_{-r}(b)}$ is a gap of $\sigma_{-r}(\S)$ with $a,b\in\S$.
\begin{lemma}\label{bnat}
	We have $b\in\N$.
\end{lemma}
\begin{proof}
	We will prove the contrapositive: if $b\in\S\sm\N$, then $\sigma_{-r}(b)$ cannot be the right endpoint of a gap. Since we are supposing that $b\not\in\N$, there must be some prime $p$ with $v_p(b)=\infty$. Fix any $\eps>0$ and let $\delta=(\sigma_{-r}(b)-\eps)/\sigma_{-r}(b)<1$. Note that
	$$\frac{\sigma_{-r}(p^\alpha)}{\sigma_{-r}(p^\infty)}\to1\Text{as}{\sspacing}\alpha\to\infty.$$
	so we can fix $\alpha$ large enough that $\delta<\sigma_{-r}(p^\alpha)/\sigma_{-r}(p^\infty)<1$. Now, consider $c\in\S$ such that
	$$v_q(c)=\begin{cases}\hfill v_q(b)\hfill&\text{if }q\neq p\\\hfill\alpha\hfill&\text{otherwise .}\end{cases}$$
	Then
	$$\sigma_{-r}(b)-\eps<\frac{\sigma_{-r}(p^\alpha)}{\sigma_{-r}(p^\infty)}\sigma_{-r}(b)=\sigma_{-r}(c)<a.$$
	Thus, we can approach $\sigma_{-r}(b)$ from below, meaning it cannot be the right endpoint of a gap.
\end{proof}
\begin{lemma}\label{very steinitz}
	If $p\ge P_r$, then $v_p(a)=\infty$ and $v_p(b)=0$.
\end{lemma}
\begin{proof}
	Let $n=p_1^{v_{p_1}(a)}\cdots p_{L_r}^{v_{p_{L_r}}(a)}$. Then, by Theorem \ref{interval},
	$$\sigma_{-r}(n)\cdot\sigma_{-r}(\S_{L_r})=[\sigma_{-r}(n),\sigma_{-r}(n)u_{L_r}].$$
	It is clear that $\sigma_{-r}(a)$ is in the above interval since $v_q(a)=v_q(n)$ for all $q<P_r$. Since $\sigma_{-r}(a)$ is the left endpoint of a gap, we must have that $\sigma_{-r}(a)=\sigma_{-r}(n)u_{L_r}$. Pick $m\in\S$ such that $a=nm$, so $v_q(m)=0$ for $q<P_r$ and $\sigma_{-r}(m)=u_{L_r}$. Recall that 
	$$u_{L_r}=\prod_{k=L_r+1}^\infty\sigma_{-r}\parens{p_k^\infty}.$$
	Note that $\sigma_{-r}(q^\alpha)<\sigma_{-r}(q^\infty)$ for all primes $q$ and exponents $\alpha\in\Z_{\ge0}$. Since $\sigma_{-r}(m)=u_{L_r}$, this means we must have $v_q(a)=v_q(m)=\infty$.
	
	To get that $v_q(b)=0$ for $q\ge P_r$, the proof is similar where you instead let $n=p_1^{v_{p_1}(b)}\cdots p_{L_r}^{v_{p_{L_r}}(b)}$ at the beginning.
\end{proof}
\begin{lemma}
	If $p<P_r$ and $v_p(a)<\infty$, then $v_p(a)<\frac{\log P_r}{\log p}-1$.
\end{lemma}
\begin{proof}
	Fix a prime $p<P_r$ with $v_p(a)<\infty$. For convenience, set $\alpha=v_p(a)$. Once again, let
	$$n=\prod_{m=1}^{L_r}p_m^{v_{p_m}(a)}.$$
	We know that $\sigma_{-r}(pa)>\sigma_{-r}(a)$, so $\sigma_{-r}(pa)>\sigma_{-r}(b)$. By Theorem \ref{interval},
	$$\sigma_{-r}(pn)\cdot\sigma_{-r}(\S_{L_r})=\sqbracks{\sigma_{-r}(pn),\sigma_{-r}(pn)u_{L_r}}.$$
	From Lemma \ref{very steinitz}, we know that $\sigma_{-r}(a)=\sigma_{-r}(n)u_{L_r}$, so $\sigma_{-r}(pa)=\sigma_{-r}(pn)u_{L_r}$. Hence, $\sqbracks{\sigma_{-r}(pn),\sigma_{-r}(pa)}\subset\sigma_{-r}(\S)$, so we must have $\sigma_{-r}(pn)>\sigma_{-r}(b)>\sigma_{-r}(a)$. Dividing both sides by $\sigma_{-r}(n)$, we arrive at
	$$\frac{\sigma_{-r}(p^{\alpha+1})}{\sigma_{-r}(p^\alpha)}>\prod_{m=L_r+1}^\infty\sigma_{-r}(p_m^\infty)=\sum_{d\in\N_{L_r}}\frac1{d^r},$$
	where $\N_{L_r}=\S_{L_r}\cap\N$ is the set of natural numbers not divisible by primes that are at most $p_{L_r}$. We can simplify these inequalities to see that
	$$p^{-r(\alpha+1)}>\frac{p^{-r(\alpha+1)}}{1+p^{-r}+\dots+p^{-r\alpha}}>\sum_{d\in\N_{L_r}\sm\{1\}}\frac1{d^r}.$$
	Since $P_r\in\N_{L_r}\sm\,\{1\}$, the above forces $P_r>p^{\alpha+1}$, from which we get that $$\frac{\log P_r}{\log p}-1>\alpha,$$
	as desired.
\end{proof}
\begin{lemma}
	If $p<P_r$, then $v_p(b)<\frac{\log P_r}{\log p}$.
\end{lemma}
\begin{proof}
	Fix a prime $p$ with $v_p(b)>0$, and note that $v_p(b)<\infty$ by Lemma \ref{bnat}. Write $\alpha=v_p(b)$. We have,
	$$\sigma_{-r}(b/p)<\sigma_{-r}(b)\implies\sigma_{-r}(b/p)<\sigma_{-r}(a)\implies\sigma_{-r}(b/p)u_{L_r}<\sigma_{-r}(a)<\sigma_{-r}(b/p\cdot p)$$
	where the second implications follows from Theorem \ref{interval} telling us that $\sigma_{-r}(b/p)\cdot\sigma_{-r}(\S_{L_r})=\sqbracks{\sigma_{-r}(b/p),\sigma_{-r}(b/p)u_{L_r}}$ is a connected subset of $\sigma_{-r}(\S)$.
	We now divide by $\sigma_{-r}(b/p)$ to get
	$$\frac{\sigma_{-r}(p^{\alpha})}{\sigma_{-r}(p^{\alpha-1})}>\prod_{m=L_r+1}^\infty\sigma_{-r}(p_m^\infty)=\sum_{d\in\N_{L_r}}\frac1{d^r},$$
	from which the desired result follows.
\end{proof}

\section{Open Problems}
There are still many questions left to answer regarding these divisor functions. For example, little is known about how often, if ever, Zubrilina numbers appear:
\begin{conject}[\cite{coin3}, Conjecture 4.1]
	There are infinitely many Zubrilina numbers.
\end{conject}
In \cite{coin4}, Defant investigates the unitary divisor functions $\ast\sigma_{-r}:\N\to\R$, which are the unique multiplicative functions satisfying $\ast\sigma_{-r}(p^\alpha)=1+p^{-r\alpha}$ for prime $p$. He conjectures that the number of connected components of $\overline{\ast\sigma_{-r}(\N)}$ is monotone in $r$, and here we make the analogous conjecture for regular divisor functions.
\begin{conject}[Divisor Function Monotonicity]
	$C_r$ is monotone in $r$.
\end{conject}
\begin{rem}
	A result that gaps ``do not close'' would resolve both of the above conjectures in the affirmative. One may even hope to show the ``pointwise'' version of this statement: that $\sigma_{-r}(a)=\sigma_{-r}(b)$ for at most one value of $r>1$ when $a\neq b$. However, this pointwise statement is false in general. In particular, if one takes
	\begin{align*}
		a &= 2503\cdot5003\\
		b &= 2467\cdot6337\cdot10007,
	\end{align*}
	then $\sigma_{-r}(a)=\sigma_{-r}(b)$ and $\sigma_{-s}(a)=\sigma_{-s}(b)$ for $r\approx1.9502$ and $s\approx3.1446$. One can find examples like this by taking vectors such that distinct $L_p$-norms are equal (in this case, we picked $L_2$ and $L_3$), scaling these vectors down sufficiently, and approximating their entries with reciprocals of nearby primes. If the monotonicity conjecture were false, we believe this would be a good way to construct a counterexample.
\end{rem}

\subsection*{Acknowledgments}  This work was completed during the University of Minnesota Duluth's Undergraduate Mathematics Research Program, and was supported by NSF/DMS grant 1650947 and NSA grant H98230-18-1-0010. We would like to thank Joe Gallian for his support and for providing a spectacular work environment. In addition, we also thank Colin Defant and Mitchell Lee for providing helpful comments and suggestions. 

\bibliographystyle{acm}

\begin{thebibliography}{1}
	
	\bibitem{coin1}
	{\sc Defant, C.}
	\newblock {On the Density of Ranges of Generalized Divisor Functions}.
	\newblock {\em Notes on Number Theory and Discrete Mathematics \bf 21\/}
	(2015), 80--87.
	
	\bibitem{coin3}
	{\sc Defant, C.}
	\newblock Connected components of complex divisor functions.
	\newblock {\em Journal of Number Theory \bf 190\/} (2018), 56--71.
	
	\bibitem{coin4}
	{\sc Defant, C.}
	\newblock Ranges of unitary divisor functions.
	\newblock {\em Integers \bf 18\/} (2018).
	
	\bibitem{laat}
	{\sc Laatsch, R.}
	\newblock Measuring the abundancy of integers.
	\newblock {\em Mathematics Magazine \bf 59}, 2 (1986), 84--92.
	
	\bibitem{sann}
	{\sc Sanna, C.}
	\newblock On the closure of the image of the generalized divisor function.
	\newblock {\em Uniform Distribution Theory (to appear)\/} (01 2018).
	
	\bibitem{Steinitz1910}
	{\sc Steinitz, E.}
	\newblock Algebraische theorie der k\"orper.
	\newblock {\em Journal f\"ur die reine und angewandte Mathematik \bf 137\/}
	(1910), 167--309.
	
	\bibitem{math}
	{\sc {Wolfram Research{,} Inc.}}
	\newblock Mathematica, {V}ersion 11.0.
	\newblock Champaign, IL, 2016.
	
	\bibitem{zubr}
	{\sc {Zubrilina}, N.}
	\newblock {On the Number of Connected Components of Ranges of Divisor
		Functions}.
	\newblock {\em arXiv:1711.02871\/} (Nov. 2017).
	
\end{thebibliography}

\end{document}